\documentclass[preprint]{elsarticle}
\usepackage{amssymb,amsfonts}
\usepackage{amsthm, amsmath}
\usepackage{natbib}
\usepackage{graphicx, color, marginnote, multirow, hyperref, caption}
\usepackage{url}
\usepackage[active]{srcltx}
\usepackage[yyyymmdd]{datetime}
\usepackage{tikz} 
\usetikzlibrary{calc}
\numberwithin{equation}{section}

\newtheorem{theorem}{Theorem}[section]

\newtheorem{proposition}[theorem]{Proposition}
\newtheorem{definition}[theorem]{Definition}
\newtheorem*{theorem*}{Theorem}

\theoremstyle{remark}
\newtheorem{remark}[theorem]{\bf{Remark}}

\newcommand{\pbar}{\overline p}

\newcommand{\ah}{\alpha}

\newcommand{\n}{\nabla}
\newcommand{\g}{\gamma}
\newcommand{\G}{\Gamma}
\newcommand{\p}{\partial}

\newcommand{\Rme}{{ \rm Re}}

\begin{document}
  
   \begin{frontmatter}
 \title{Productivity Index for Darcy and pre-/post-Darcy Flow (Analytical Approach)}

\author{Lidia Bloshanskaya\corref{cor1}}
\ead{bloshanl@newpaltz.edu}
\address{SUNY New Paltz, Department of Mathematics, 1 Hawk dr, New Paltz, NY 12561}
\author{Akif Ibragimov}
\ead{akif.ibraguimov@ttu.edu}
\address{Texas Tech University, Department of Mathematics and Statistics, Broadway and Boston, Lubbock, TX 79409-1042}
\author{Fahd Siddiqui}
\ead{fahd.siddiqui@ttu.edu}
\address{Texas Tech University, Bob L. Herd Department of Petroleum Engineering, Box 43111, Lubbock, TX 79409-3111}
\author{Mohamed Y. Soliman}
\ead{ mohamed.soliman@ttu.edu}
\address{Texas Tech University, Bob L. Herd Department of Petroleum Engineering, Box 43111, Lubbock, TX 79409-3111}

\cortext[cor1]{Corresponding author}

\begin{abstract}
   We investigate the impact of nonlinearity of high and low velocity flows on the well productivity index (PI). Experimental data shows the departure from the linear Darcy relation for high and low velocities. High-velocity (post-Darcy) flow occurring near wells and fractures is described by Forchheimer equations and is relatively well-studied. While low velocity flow receives much less attention, there is multiple evidence suggesting the existence of pre-Darcy effects for slow flows far away from the well. This flow is modeled via pre-Darcy  equation. We combine all three flow regimes, pre-Darcy, Darcy and post-Darcy, under one mathematical formulation dependent on the critical transitional velocities. This allows to use our previously developed framework to obtain the analytical formulas for the PI for the cylindrical reservoir. We study the impact of pre-Darcy effect on the PI of steady-state flow depending on the well-flux and the parameters of the equations.
\end{abstract}


   \end{frontmatter}


%


\section{Introduction}\label{sec:intro}

Darcy's pioneering work is at the heart of studies on fluid filtration in porous media. Based on experimental observations, Darcy's law linearly relates the fluid velocity to the pressure gradient. While it is commonly used by engineers, the limitation of it's applicability is generally accepted.  The linear model fails to take into account and explain phenomena occurring beyond the range of moderate velocity flow. 

The high-velocity bound  was recognized even by Darcy himself, \cite{Darcy}. The fast flows near wells and fractures are often described as post-Darcy, and can be described by variety of models. One of the most used is the model suggested by Dupuit and Forchheimer in \cite{Forchheimer, Dupuit} who proposed three different polynomial equations, two- and three-term laws and power law, to match the experimental data, see also \cite{Straughan-porosity}. We will refer to the flow in high-velocity zone as post-Darcy or Forchheimer flow.



The low-velocity bound of validity of Darcy equation did not  receive the widespread attention among researchers and engineers. Multiple experimental results, however, point out the existence of the non-Darcy effects for slow flow, \cite{Fishel, Dudgeon, Soni}. We will refer to such flow as pre-Darcy or pre-laminar, \cite{basak77}.  It occurs in media with low permeability and high specific surface area and for the high-viscosity flow when the fluid can exhibit non-Newtonian properties, see \cite{Wen2006, Bagci14, longmuir, fand1987, Basniev-rus} and references therein.   Pre-Darcy effect constitutes the ``greater than proportional'' increase in fluid velocity $v$ with respect to pressure gradient $\n p$, and is described by power-law equation $|v|^{1-s}\sim|\n p|$ for $0\leq s< 1$, \cite{Wen2006, Bagci14}. The case $s=0$ corresponds to Darcy flow. For extensive overview of experimental results on pre-Darcy effect see, for example, \cite{longmuir} and references therein.

Following \cite{basak77} the relation between the pressure gradient and fluid velocity depending on the flow regime is shown on Fig.~\ref{fig:nolinearity}.

The petroleum reservoir is an example of environment exhibiting all three flow regimes: fast flow near the well (post-Darcy/Forchheimer), slow flow near the outer boundaries of the reservoir (pre-Darcy) and moderate flow in between (Darcy). The understanding of pre-Darcy effects can be especially important to improve the well management for the low rate oil reservoir, operating predominantly in pre-Darcy zone. This motivates us to study the impact of  combined non-Darcy effects on the value of Productivity Index (PI), the dimensionless characteristic of well capacity, \cite{Dake, Raghavan, Slider}. The PI is defined as a ration of the well-bore  flux, and the difference between the average pressures in the reservoir and on the well-bore. It is used to quantify the ability of the reservoir to deliver fluids to the well-bore: the better is the performance of the well, the higher is the value of the PI. While in general PI is  time-dependent it exhibits the stabilization property: once the well production is stabilized (i.e. the well-flux becomes constant) the PI becomes constant independent from the production  history or even the operating conditions (see references above). The corresponding flow regime is called pseudo-steady state (PSS)  and PSS pressure, velocity and PI serve as pseudo-attractor to the transient one, \cite{Raghavan, AIVW09, ABI12, ABI15}. From reservoir engineering point of view the PSS regime is attained at the time when the perturbation from the well reaches the exterior boundary  of the well  drainage area. Therefore engineers use the computed PSS PI  to estimate the size of the area impacted by the well.

In \cite{AIVW09, ABI11} we studied PI for post-Darcy (fast) flow extensively. This paper is primarily dedicated to the model of pre-Darcy coupled with post-Darcy and Darcy flows in cylindrical reservoir. In addition in this paper we use some parameters obtained from our own experimental setup. Three flow zones (fast, moderate and slow) are considered and the transition between the zones is characterized by the critical velocities. The flow in each zone can be described by either of three regimes: Darcy, pre-, post-Darcy.  We use previously developed framework to obtain the analytical formulas for PI in different flow descriptions. 
We illustrate the impact of pre-Darcy effect depending on the well production rate, power $s$ in pre-Darcy relation and transitional velocities. 
Our results show considerable impact of the pre-Darcy effect on the overall PI of the well, especially for smaller production rate/flux and larger values of the power $s$.
The results can be of especial interest for improvement economic recovery of hydrocarbons in low-rate reservoirs.

The paper is organized as follows. Sec.~\ref{sec:probl-statement} addresses the mathematical and physical background of non-Darcy flow in porous media and the concept of PI. In Sec.~\ref{sec:analytical} the analytical  formulas for PI for axial-symmetric flows are presented. Sec.~\ref{sec:comp-results} is devoted to the computational results and their analysis. Finally, Sec.~\ref{sec:experiment} provides the results of our experiments aimed to  estimate the extent of the pre-Darcy effect (the value of power $s$) and the range of its significance (velocity range).

\section{Background and problem description}\label{sec:probl-statement}
Consider a bounded reservoir domain $U$ with the internal boundary $\G_w$ (well surface) and external impermeable boundary $\G_e$. We consider the flow of slightly compressible fluid characterized by the equation of state, \cite{Muskat} 
\begin{equation}\label{eq-state}
 \frac{\partial \rho}{\partial p}=\gamma\rho, \quad\gamma\sim10^{-8},
\end{equation}
where $\rho(p,t)$ is the fluid density, $p(x,t)$ is the pressure and $\gamma$ is compressibility coefficient.
Darcy's law $v=-\alpha\n p$ describes the proportional relation between the velocity of the fluid $v$ and pressure gradient $\nabla p$. Here $\ah=\mu/k$ is Darcy coefficient, $\mu$ is the viscosity of the fluid, $k$ is the permeability of the porous media. The linear relation however holds only for certain range of velocities (and pressure gradients, correspondingly) for laminar flow. Near the well the flow is very fast since the area exposed to the flow is very small. For the velocities above some critical one $v_F$ the flow becomes post-Darcy/Forchheimer. At the same time further away from the well the velocity drops significantly and becomes pre-laminar/pre-Darcy for velocities below some critical $v_D$. To model the flow in the whole reservoir we consider the following combined equation of motion
\begin{equation}\label{eq:g-v-np}
  g(|v|)v =-\nabla p
\end{equation}
where 
\begin{align}
 &\label{eq:prd} g(\xi)=\lambda\xi^{-s},\quad 0\leq s<1, &&\text{for}\quad 0\leq\xi\leq v_D,\quad \mbox{Pre-Darcy flow;}\\ 
 &\label{eq:d}g(\xi)=\ah,&&\text{for}\quad v_D\leq \xi\leq v_F,\quad \mbox{Darcy flow;}\\
 &\label{eq:pstd}g(\xi)=\ah  +\beta \xi,  &&\text{for}\quad \xi\geq v_F,\quad \mbox{2-Forchheimer flow.}
\end{align}
Here  $\beta$ is the Forchheimer coefficient.
The resulting relation between $|v|$ and $|\n p|$ and the comparison with the linear Darcy law is depicted on Fig.~\ref{fig:nolinearity}. 
\begin{figure}[!h]
 \centering
 \includegraphics[width=0.7\textwidth]{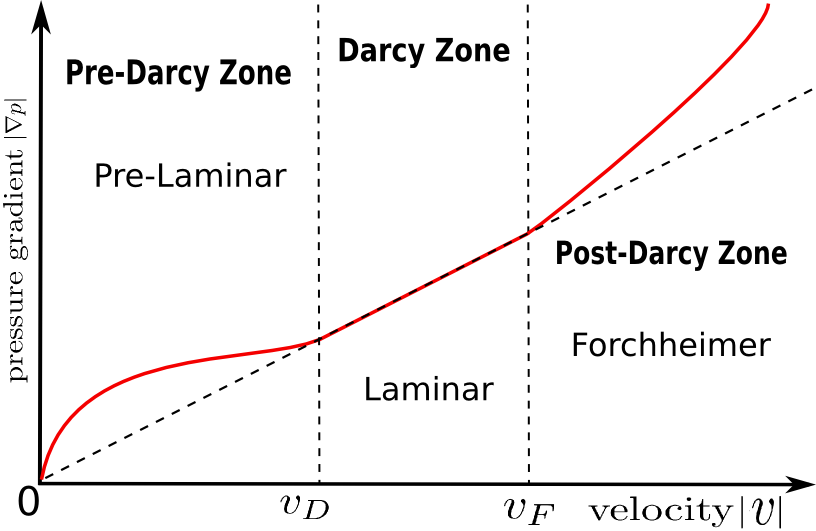}
 \caption{Deviation from linear Darcy law}
 \label{fig:nolinearity}
\end{figure}

\begin{remark}
The transition from pre-Darcy to Darcy flow and from Darcy to post-Darcy flow is often associated with the dimensionless Reynolds number ${\rm Re}$, quantifying the ration of inertial to viscous forces, \cite{Bear, Muskat}.  However, unlike  the case of the pipe flow, there is large discrepancy in the transitional values of ${\rm Re}$ for the flow in porous media (from $\Rme\sim10$ to $\Rme\sim 1000$ for transition to Forchheimer, \cite{Bagci14, Bear,krakowska14}). Instead of $\Rme$  we will use the velocity as the critical characteristic for transition between the flow regimes. 
\end{remark}

%


%

Equation \eqref{eq:g-v-np} can be rewritten in the equivalent form as an explicit formula for  $v$ through $\nabla p$:
\begin{equation}\label{eq:v-Knp}
  v=-K(|\nabla p|)\n p,
\end{equation}
where 
\begin{align}
 &K(|\n p|)=\lambda^{-\frac1{1-s}}|\n p|^{\frac{s}{1-s}},\quad 0\leq s<1, &&\text{pre-Darcy flow;}\\
 &K(|\n p|)=\ah^{-1} &&\text{Darcy flow;}\\
 &K(|\n p|)=\frac2{\ah+\sqrt{\ah^2+4\beta|\n p|}} &&\text{post-Darcy (Forchheimer) flow.}
\end{align}

Combining the continuity equation
\begin{equation}
 \frac{\partial \rho}{\partial t}+\nabla\cdot(\rho v)=0 
\end{equation}
with equation \eqref{eq-state} we obtain 
\begin{align}
  \frac{\partial \rho}{\partial t}= \frac{\partial \rho}{\partial p} \frac{\partial p}{\partial t}
  =-\rho(\nabla\cdot v)-\nabla\rho\cdot v
  =-\rho(\nabla\cdot v)-\frac{\partial \rho}{\partial p}(\nabla p\cdot v).
\end{align}
Thus
\begin{equation}\label{eq:cont+stateFull}
  \frac{\partial p}{\partial t}=-\gamma^{-1} (\nabla\cdot v)-\nabla p\cdot v.
\end{equation}
Since  in natural reservoirs $\gamma\sim10^{-8}\,{\rm Pa}^{-1}$, in engineering practice the second term is always neglected in comparison to the first one,\footnote{In case of axial-symmetric  flows it is easy to show the continuous dependence of the solution  of \eqref{eq:cont+stateFull} on $\gamma$ and its affinity to the solution of  \eqref{eq:p-1}, see Proposition~\ref{lem:vel-diff-gamma}.} \cite{Raghavan, Dake}.  In view of \eqref{eq:v-Knp} we arrive to the truncated  degenerate parabolic equation of pressure
\begin{equation}\label{eq:p-1}
  \g\frac{\partial p}{\partial t}=-(\nabla\cdot v)=\n\cdot(K(|\n p|)\n p).
\end{equation}

For truncated equation \eqref{eq:p-1} engineers use an integral characteristic of reservoir-well system 
to measure the capacity of the well, \cite{Dake, Raghavan, Slider}. 
This characteristic is called the productivity index (PI) and can be formulated as a function over solution of \eqref{eq:p-1}.
\begin{definition}
 The PI in the reservoir $U$ with boundaries $\Gamma_e$ and $\Gamma_w$ is defined as a  functional 
\begin{equation}\label{def:pi}
J_g(t)=\frac{Q(t)}{\pbar_U(t)-\pbar_{\Gamma_w}(t)}=\frac{\int_{\Gamma_w}v\cdot N\,ds}{\frac1{|U|}\int_Up(x,t)dx- \frac1{|\Gamma_w|}\int_{\Gamma_w}p(x,t)ds}.
\end{equation}
Here $Q(t)$ is the flux/production rate on the boundary $\Gamma_w$ and the difference in the denominator is called pressure-drawdown.
\end{definition}

In general, the functional PI is  time dependent. However if the production rate $Q(t)$ stabilizes over time to a constant value $Q$, then under certain conditions on the boundary data the PI as well stabilizes over time to specific constant value regardless of the initial pressure distribution (see \cite{Muskat, Raghavan, ABI11, ABI12, ABI15}). This value is determined by the special pressure distribution, called the {\it pseudo-steady state $($PSS$)$} solution such that 
\begin{equation}
 \frac{\partial p_s}{\partial t}=-\g A=const.
\end{equation}
Thus the PSS has the form
\begin{equation}
 p_s(x,t)=-\g At+B+W(x),\label{pW}
\end{equation}
where $W(x)$ is the basic profile corresponding to the constant flux $Q$ and is defined as the solution of the steady state BVP
\begin{align} 
&-\g A = \nabla \cdot K(|\nabla W|)\nabla W,\label{W_eq}\\
&\frac{\partial W}{\partial N}=0 \quad\hbox{on}\quad \Gamma_e,\label{neumann_W} \\
 &W=0 \quad\hbox{on}\quad \Gamma_w.\label{pss-d}
\end{align}
Here $A=Q/|U|$.

For the PSS case the definition \eqref{def:pi} can be written as
\begin{equation}\label{def:pi-pss}
 J=\frac{Q |U|}{\int_U W(x)\,dx}.
\end{equation}

From reservoir engineering point of view the PSS regime is attained at the time when the perturbation from the well reaches the exterior boundary  of the well  drainage area $\Gamma_e$. Therefore engineers use the computed PSS PI \eqref{def:pi-pss} to estimate the size of the area impacted by the well.

\section{Analytical PI for axial-symmetric flows}\label{sec:analytical}
In this section we obtain the analytical formulas for PI in  cylindrical reservoir in axial-symmetric flows, Fig.~\ref{fig:cyl-reserv}.

 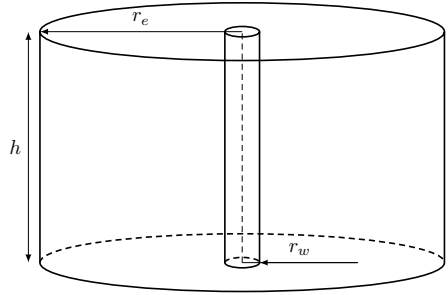
\begin{figure}[!h]
 \centering
\resizebox{0.5\textwidth}{!}{
\begin{tikzpicture}
\draw[thick] (-3.5,4) -- (-3.5,0) arc (180:360:3.5cm and 0.5cm) -- (3.5,4) ++ (-3.5,0) circle (3.5cm and 0.5cm);
\draw[densely dashed,thick] (-3.5,0) arc (180:0:3.5cm and 0.5cm);

\draw[thick] (-.3,4) -- (-.3,0) arc (180:360:.3cm and 0.09cm) -- (.3,4) ++ (-.3,0) circle (.3cm and 0.09cm);
\draw[densely dashed, thick] (-.3,0) arc (180:0:.3cm and 0.09cm);
\draw[densely dashed] (0,0) -- (0,4);

\draw[>=latex,<->] (-3.7,0) -- node[left]{$h$} (-3.7,4) ;
\draw[>=latex,<-] (-3.5,4) -- node [above]{$r_e$} (0,4) ;
\draw (0,0) -- node [above]{$r_w$} (2,0) ;
\draw[>=latex,<-] (0.3,0) --  (2,0) ;
\end{tikzpicture} }
\caption{Cylindrical reservoir   with thickness $h$, radius $r_e$ with a fully penetrated vertical well of radius $r_w$ positioned in its center}
\label{fig:cyl-reserv}
\end{figure}


First we show the affinity of solutions of full and truncated equations \eqref{eq:cont+stateFull} and \eqref{eq:p-1} for small values of compressibilty $\gamma$. For axial-symmetric pseudo-steady flows the velocity is a scalar and equations \eqref{eq:cont+stateFull} and \eqref{eq:p-1}  will take the form 
 \begin{align}
  &\g A=(\nabla\cdot v_\g)+\g\nabla p\cdot v_\g=\frac1r\frac{\partial}{\partial r}(r v_\g)-\g g(v_\g)v_\g^2,\quad r_w<r<r_e,\label{eq:cont+pss-radial}\\
   &\g A=\nabla\cdot v=\frac1r\frac{\partial}{\partial r}(r v),\quad r_w<r<r_e\label{eq:trunc-radial}
\end{align}
with boundary conditions
\begin{equation}\label{bc-0}
v(r_e)=v_{\g}(r_e)=0.
\end{equation}

\begin{proposition}\label{lem:vel-diff-gamma}
 The difference $|v_\g-v|\leq \g C$, where $C$ depends on the reservoir radius $r_e$ only.
\end{proposition}
\begin{proof}
Obviously, $|v_\g|\leq C_0$ where $C_0$ does not depend on $\g$ for $\g<1$.
Subtracting \eqref{eq:cont+pss-radial} and \eqref{eq:trunc-radial} we get
\begin{equation}
   \frac{\p}{\p r}(r(v_\g-v))=\g g(v_\g)v_\g^2\leq \g C, \quad (v_\g-v)|_{r=r_e}=0.
\end{equation}
Integrating in $r$ the formula above one gets the result.
\end{proof}
From now on we consider the truncated equation \eqref{eq:trunc-radial} and for simplicity $\g=1$.

For cylindrical reservoir  
\begin{equation}\label{A which is T}
 A=\frac{Q}{|U|}=\frac{Q}{2\pi h (r^2_e-r^2_w)}.
\end{equation}
Integrating \eqref{eq:trunc-radial} in view of condition \eqref{bc-0} one gets a solution 
\begin{equation}\label{def:rad-vel}
 v(r)=A\,(r_e^2-r^2)r^{-1}=\frac{Q(r_e^2-r^2)}{2\pi h r(r^2_e-r^2_w)},\quad r_w<r<r_e.
\end{equation}


%

Solving for the radius $r$ gives
\begin{align}
  r(v)=\frac{-\pi (r^2_e-r^2_w) v+ \sqrt{\pi^2v^2 (r^2_e-r^2_w)^2 +(Q/h)^2r_e^2}}{Q/h}.\label{r-through-v}
\end{align}
Then using \eqref{r-through-v} we obtain the critical values of the radius $r_D=r(v_D)$ and $r_F=r(v_F)$ depending on the critical velocities $v_D$ and $v_F$.
 We consider three flow zones with respect to the velocity of the flow, see Fig.~\ref{fig:flow-zones}:
\begin{itemize}
 \item red: near the well $r_w\leq r\leq r_F$ with the fast flow $ v_F\leq v(r)\leq v(r_w)=q/(2\pi h r_w)$;
 \item blue: middle of the reservoir $r_F\leq r\leq r_D$ with moderate flow $ v_D\leq v(r)\leq v_F$;
\item green: near the outer boundary of the reservoir $r_D\leq r\leq r_e$, with the slowest flow $ 0\leq v(r)\leq v_D$.
\end{itemize}
\begin{figure}
   \centering
  \includegraphics[width=0.8\textwidth]{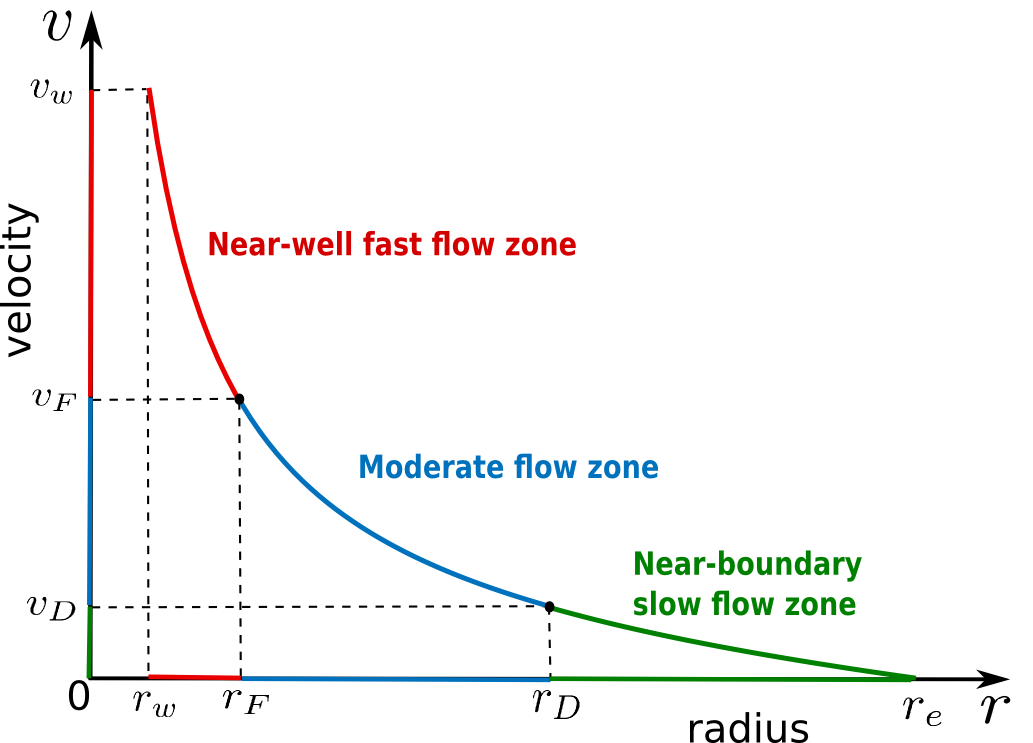}
  \caption{Flow zones, depending on critical velocities and radii}
  \label{fig:flow-zones}
\end{figure}

In order to investigate the  pre- and post-Darcy effects of the flow we introduce  five flow regimes: Darcy  in all three zones; Forchheimer  in all three zones; Forchheimer in the near-well zone, Darcy in the middle and near boundary zones; Darcy in near-well and middle zones and pre-Darcy in near-boundary zone; Forchheimer in near-well, Darcy in the middle and pre-Darcy in the near-boundary zones. For convenience we use the following description of each regime presented below in the table. 
\begin{center}
 \begin{tabular}{c  || c | c| c }
 \multirow{2}{*}{notation} & near-well zone  & middle zone & near-boundary zone \\
 &(red)&(blue)&(green)\\
 \hline\hline
 Darcy (D) &Darcy&Darcy&Darcy\\
 Forch (F)&Forchheimer&Forchheimer&Forchheimer\\
 FDD&Forchheimer&Darcy&Darcy\\
 DDpD&Darcy&Darcy&preDarcy\\
  FDpD&Forchheimer&Darcy&preDarcy\\
 \end{tabular}
  \medskip
\end{center}

For $A$ as in \eqref{A which is T} denote
 \begin{align}
  S_D[r_1,r_2]&
  =\ah \int_{r_1}^{r_2}(r_e^2-r^2)^2r^{-1}\,dr,\label{I_D}\\
 S_F[r_1,r_2]&
 =\int_{r_1}^{r_2}\!\!\left[\ah+\beta A\, (r_e^2-r^2)r^{-1}\right](r_e^2-r^2)^2r^{-1}\,dr,\label{I_F}\\
 S_{pD}[r_1,r_2]&
 =\lambda A^{-s}\int_{r_1}^{r_2}(r_e^2-r^2)^{2-s}r^{s-1}\,dr.\label{I_pD}
\end{align}
The PI for different regimes can be expressed via the analytical formulas below
\begin{proposition}\label{prop:analyticPI}
Let $L=2\pi h(r_e^2-r_w^2)^2$. Then\\
the PI for Darcy flow in all three zones is
\begin{equation}
   J_{D} =\dfrac{L
 }{S_D[r_w,r_e]};
\end{equation}
the PI for Forchheimer flow in all three zones is
\begin{equation}
    J_{F} =\dfrac{L }{S_F[r_w,r_e]};
\end{equation}
the PI for Forchheimer-Darcy-Darcy flow  is
\begin{equation}
  J_{FDD} =\dfrac{L}{S_F[r_w,r_F]+S_D[r_F,r_e]}; 
\end{equation}
the PI for Darcy-Darcy-pre-Darcy flow is
\begin{equation}
    J_{DDpD} =\dfrac{L}{S_D[r_w,r_D]+S_{pD}[r_D,r_e]};
\end{equation}
the PI for Forchheimer-Darcy-pre-Darcy flow  is
\begin{equation}
   J_{FDpD} =\dfrac{L}{S_F[r_w,r_F]+S_D[r_F,r_D]+S_{pD}[r_D,r_e]}.
\end{equation}
\end{proposition}
\begin{proof}
 First, formula \eqref{def:pi-pss} for the PSS PI can be rewritten in terms of velocity $v$. Since $\n p_s=\n W$ we have 
\begin{align}
\frac Q{|U|}&\int_U W(x)\,dx=\int_U K(|\n p|)|\n p|^2\,dx
=\int_U |v||\n p|\,dx=\int_U g(|v|)|v|^2\,dx.
\end{align}

Then the PI in the reservoir can be calculated as (since $v$ is scalar)
\begin{align}\label{pi-analytic0}
 J=\frac{Q^2}{\int_U g(v)v^2\,dx}=\frac{Q^2}{2\pi h\int_{r_w}^{r_e} r g(v)v^2\,dr}.
 \end{align}
 Using the definitions \eqref{eq:prd}-\eqref{eq:pstd} of $g(v)$ corresponding to the flow and formulas \eqref{def:rad-vel} for velocity one gets the formulas for PI.
\end{proof}

So-called ``skin factor'' is often used by petroleum engineers to analytically  model  the difference from the pressure drop predicted by Darcy's law and  observed data, \cite{schlum}. For this purpose we also reformulate the Proposition~\ref{prop:analyticPI} 
\begin{proposition}\label{prop:analyticPI-1}
Non-Darcy PI can be found from the Darcy PI:\\
Forchheimer flow in all three zones
\begin{equation}
 {J_F} =J_{D}\,\dfrac{S_D[r_w,r_e]}{S_F[r_w,r_e]};
\end{equation}
Forchheimer-Darcy-Darcy flow 
\begin{equation}
{J_{FDD}} =J_{D}\,\dfrac{S_D[r_w,r_e]}{S_F[r_w,r_F]+S_D[r_F,r_e]}; 
\end{equation}
Darcy-Darcy-pre-Darcy flow
\begin{equation}\label{skin-DDpD}
 {J_{DDpD}} =J_{D}\,\dfrac{S_D[r_w,r_e]}{S_D[r_w,r_D]+S_{pD}[r_D,r_e]};
\end{equation}
Forchheimer-Darcy-pre-Darcy flow 
\begin{equation}\label{skin-FDpD}
   {J_{FDpD}} =J_{D}\,\dfrac{S_D[r_w,r_e]}{S_F[r_w,r_F]+S_D[r_F,r_D]+S_{pD}[r_D,r_e]}.
\end{equation}
\end{proposition}

The equations can be used to dead-zone flow in case of time-dependent problem. Namely, \cite{kamin-vazquez} showed that there exists a solution for initial value Cauchy problem such that at each fixed moment of time $t$ the support of this solution is bounded, and as $t\to\infty$ the support is also increasing to infinity. Since the PSS regime is attained when the well perturbation reaches the outer reservoir boundary, the obtained skin  formulas \eqref{skin-DDpD} and \eqref{skin-FDpD} can be used to calculate the effective radius ($r_e$)   predicted by Darcy assumption only.


%

\section{Analysis of Computational Results}\label{sec:comp-results}

Using the formulas for the PI in Proposition~\ref{prop:analyticPI} we illustrate the influence of the nonlinearity of the flow regime on the value of the PI of the well.
Note, that all the values of the PI are dimensionless, obtained by multiplying the computational results by the factor $\ah/(2\pi h)$, ${\rm (Pa \cdot sec)/m^3}$.
The particular values of the parameters are presented below. \\
The geometrical parameters
\begin{itemize}
  \item  $r_e= 1000$ m - the radius of the reservoir
  \item  $r_w=0.3$ m - the radius of the well
  \item $h=10$ m - the reservoir thickness
\end{itemize}
The hydrodynamical parameters
\begin{itemize}
 \item $\ah=1.01\times10^{10}\ {\rm (Pa \cdot sec)/m^2}$
 \item $\lambda=1.01\times10^{10}\  ({\rm Pa} \cdot {\rm sec}^{1-s})/{\rm m}^{2-s}$
 \item $\beta=2.4318\times10^{11}\  {\rm (Pa \cdot sec^2)/m^3}$
 \item $Q/h=10^{-4}\ {\rm m^2/sec}$
 \item $s=0.3, 0.5, 0.7$
\end{itemize}
The critical velocities
\begin{itemize}
 \item $v_F=10^{-5}\ {\rm m/sec}$ - transition between Forchheimer and Darcy flow
 \item $v_D=10^{-7}\ {\rm m/sec}$ - transition between Darcy and preDarcy flow
\end{itemize}

\subsection{Dependence of PI on flux $Q/h$}
The dependence of PI on the value of flux $Q/h$ is presented in Table~\ref{tab:PIvsQ} and Fig.~\ref{fig:vsQ-gen} - \ref{fig:vsQ-maxPI}. The graphs are obtained for power $s=0.7$.

The Darcy PI  does not depend on the flux $Q/h$, so if  in all three flow zones the flow is described by Darcy law, the PI stays constant (blue curve on graphs, the Darcy column in Table). In case of post-Darcy (Forchheimer) flow, the PI dramatically decreases with increasing flux $Q/h$ (red curve, column Forch). One can see that in the presence of Forchheimer description (red, yellow, green curves and columns Forch, FDD, FDpD) its impact on PI increases with the increasing flux. For large values of $Q/h$ the PI with post-Darcy becomes almost zero, see Fig.~\ref{fig:vsQ-gen} and the last row ($Q/h=10^4\ {\rm m^2/sec}$) in Table~\ref{tab:PIvsQ} (columns Forch, FDD and FDpD). The obtained result indicate that the higher values of $Q/h$ correspond to higher velocity \eqref{def:rad-vel} and, consequently, to larger value of transitional radius $r_F$ in \eqref{r-through-v} and hence to wider ``fast flow'' zone.
In the absence of Forchheimer component (DDpD) the values of PI stabilize to constant Darcy value, since $r_D$ is increasing and pre-Darcy zone is decreasing, see purple curve on Fig.~\ref{fig:vsQ-gen} and \ref{fig:vsQ-maxPI} and column DDpD.

\begin{figure}
  \centering
  \includegraphics[width=0.7\textwidth]{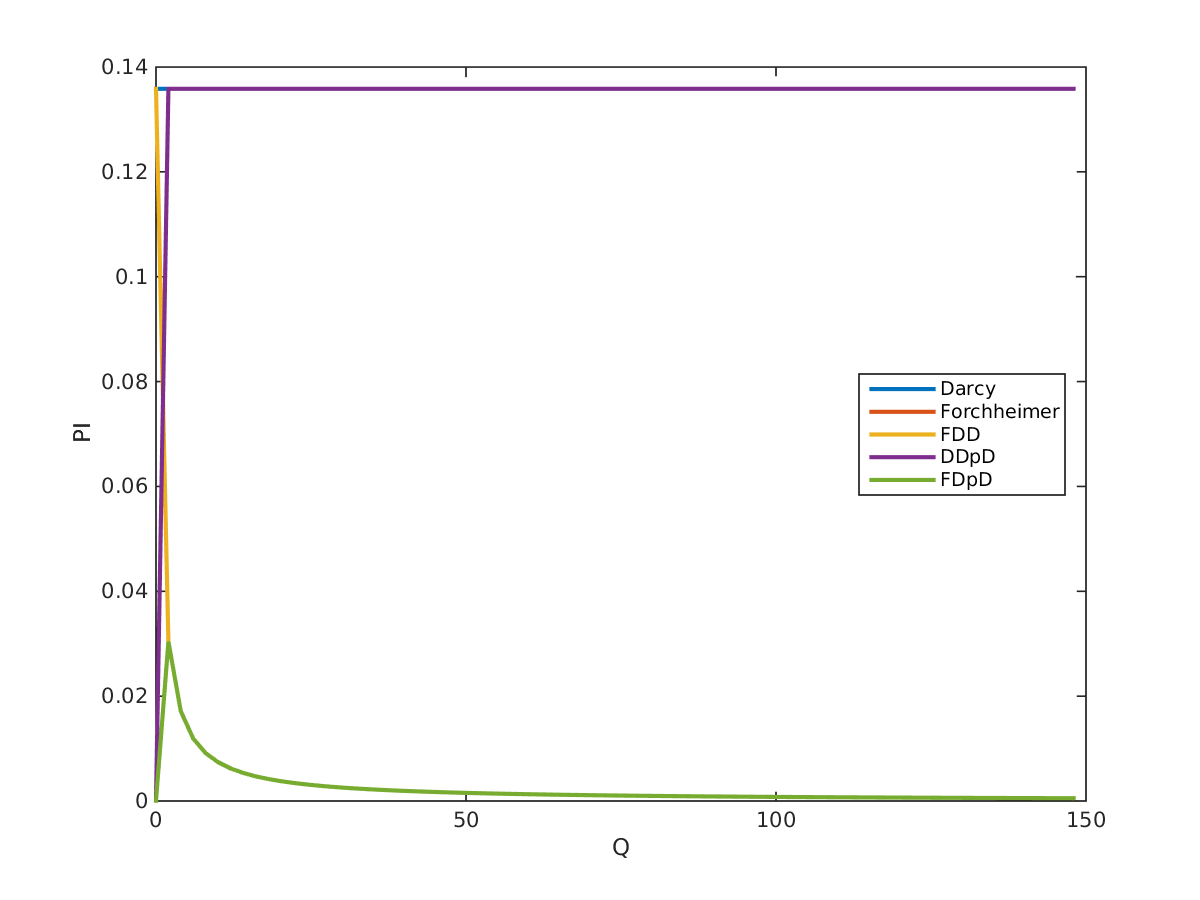}
  \vspace{-0.6cm}
  \caption{Dependence on $Q/h$, $s=0.7$ }
  \label{fig:vsQ-gen}
\end{figure}

For small values of flux the situation is the opposite: the predominant flow regime is pre-Darcy, since  the transition radius $r_D$ is decreasing with decreasing flux, making the slow-flow zone larger. Fig.~\ref{fig:vsQ-smallQ} illustrates that while the PIs for Forchheimer and FDD descriptions are almost identical to Darcy for small values of flux, the  DDpD and FDpD PIs are almost zero for small values of $Q/h$. The exact values can be seen in the corresponding columns in Table~\ref{tab:PIvsQ} for $Q/h=2\cdot10^{-7}, 10^{-4}\ {\rm m^2/sec}$. The pre-Darcy effect is greater for larger values of power $s$.
  \begin{figure}
    \centering
  \includegraphics[width=0.7\textwidth]{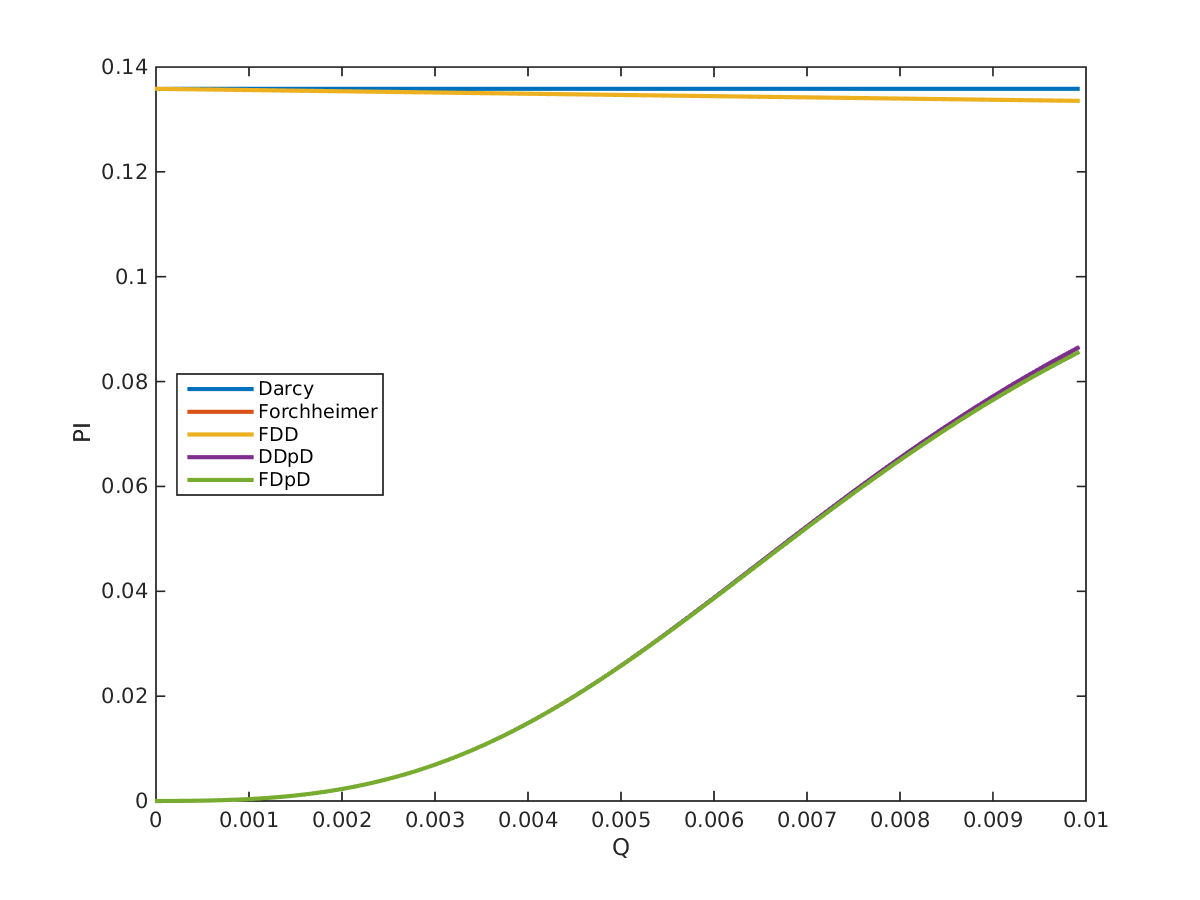}
  \vspace{-0.6cm}
  \caption{Dependence on $Q/h$, for small values of flux resulting in dominating pre-Darcy effects, $s=0.7$}
  \label{fig:vsQ-smallQ}
\end{figure}
Fig.~\ref{fig:vsQ-maxPI} shows that  FDpD PI   increases along with DDpD PI (predominant pre-Darcy for small $Q/h$), reaches the maximum value, and starts to decrease along with FDD case (predominant Forchheimer for large $Q/h$). The corresponding maximum values will depend  on the value of power $s$ in pre-Darcy case and are presented in Table~\ref{tab:PIvsQ}.
\begin{figure}
  \centering
  \includegraphics[width=0.7\textwidth]{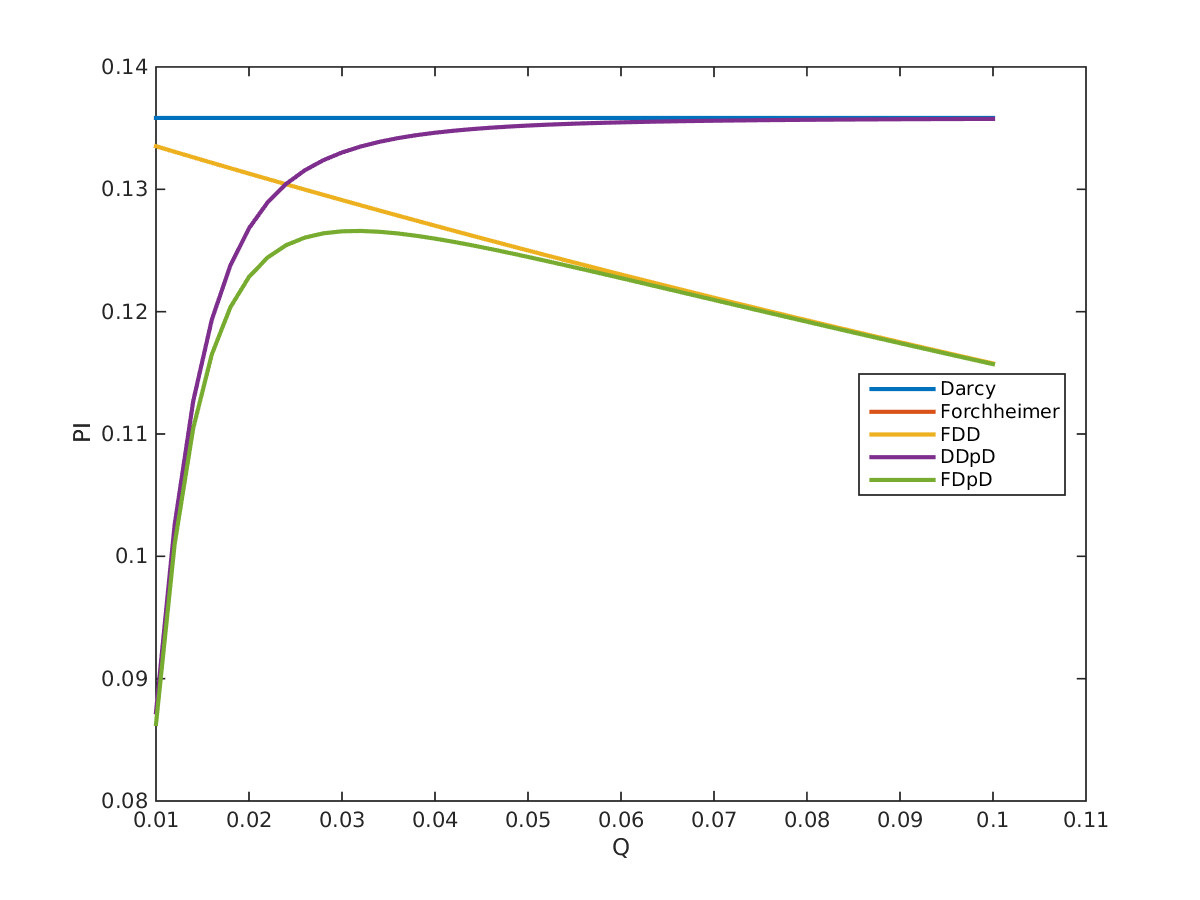}
  \vspace{-0.6cm}
  \caption{Dependence on $Q/h$, FDpD PI reaches maximum value, $s=0.7$}
 \label{fig:vsQ-maxPI}
\end{figure}

  \begin{table*}
 \centering
 \begin{tabular}{c  | c| l l l l| c}
    flux $Q$ & Darcy & Forch & FDD & DDpD & FDpD& power $s$\\
  \hline\hline
 \multirow{2}{*}{2e-7}
      & \multirow{15}{*}{0.1358}  &  \multirow{2}{*}{0.1358 }  & \multirow{2}{*}{0.1358 } &	 \multicolumn{2}{c|}{4.5e-8}  & ${ 0.7}$\\
        & &  &  & \multicolumn{2}{c|}{2.9e-4}& $0.3$ \\\cline{3-7}\cline{1-1}
 \multirow{2}{*}{ 1e-4}
      & &  \multirow{2}{*}{0.1358}  & \multirow{2}{*}{0.1358}  & \multicolumn{2}{c|}{5.65e-6} & $0.7$\\ 
     & &   &    & \multicolumn{2}{c|}{5.21e-3}   & $0.3$ \\ \cline{3-7}\cline{1-1}
 \multirow{2}{*}{ 1e-3} 
      &  & \multirow{2}{*}{0.1356 }  & \multirow{2}{*}{0.1356 }  & 3.64e-4   & 3.64e-4 & $0.7$ \\  
      &  &    &   & 9.14e-2   & 9.13e-2&  $0.3$  \\  \cline{3-7}\cline{1-1}
  5.95e-3
     & &  0.1345  & 0.1345  & 0.1354   & 0.134 &$0.3$  \\  \cline{3-7}\cline{1-1}
 \multirow{2}{*}{ 1e-2}
       &    & \multirow{2}{*}{0.1335 }  & \multirow{2}{*}{0.1335 }  & 0.0873  & 0.0863 & $0.7$  \\  
      &   &    &    & 0.1357     & 0.1334  & $0.3$   \\  \cline{3-7}\cline{1-1}
 3.18e-2
       &   & 0.1287   & 0.1287   & 0.1335    & 0.1266  & $0.7$ \\\cline{3-7}\cline{1-1}
 \multirow{2}{*}{ 1e-1}
       &   &\multicolumn{2}{|c}{ \multirow{2}{*}{0.1158}}    & 0.1358    & 0.1157 & $0.7$  \\  
        &  &     &     & 0.1358     & 0.1158 & $0.3$  \\  \cline{3-7}\cline{1-1}
 1   
         &  & \multicolumn{2}{|c}{ 0.0497}   & 0.1358    & 0.0497  & $0.7, 0.3$ \\ \cline{3-7}\cline{1-1}
 1e1  
        &   & \multicolumn{2}{|c}{ 7.41e-3}   & 0.1358    & 7.41e-3& $0.7, 0.3$   \\  \cline{3-7}\cline{1-1}
 1e4 
     &  & \multicolumn{2}{|c}{ 7.84e-6}   & 0.1358    & 7.84e-6  & $0.7, 0.3$   
 \end{tabular} 
 \vspace{.1cm}
 \caption{PI for different flow regimes vs. flux $Q$ for $s=0.7$, $s=0.3$ }
 \label{tab:PIvsQ}
\end{table*}

\subsection{Dependence of PI on pre Darcy power $s$}
The impact of power $s$ in pre-Darcy relation \eqref{eq:prd} on the values of PI for various values of $Q/h$ and $v_D$ is presented in Tables~\ref{tab:PIvsS-Q} and \ref{tab:PIvsS-vD} and Fig.~\ref{fig:vsS-vd-6} -   \ref{fig:vsS-vd-7-Q-2}. Everywhere in this section  $v_F=10^{-5}{\rm m/sec}$. 

If $s=0$ the pre-Darcy relation \eqref{eq:prd} is identical to Darcy law \eqref{eq:d}, so the FDpD and DDpD PI are equal to Darcy PI. Figures \ref{fig:vsS-vd-6} -   \ref{fig:vsS-vd-7-Q-2} show that the smaller is the transitional velocity $v_D$ to pre-Darcy regime (i.e. the smaller is the slow green zone), the longer these PIs stay the same.   As $s$ increases, the FDpD and DDpD PI drop  to 0, as the pre-Darcy effect increases. For the considered value of flux $Q/h=10^{-4}{\rm m^2/sec}$ the values of DDpD and FDpD PIs are almost the same, see the corresponding lines of Table~\ref{tab:PIvsS-Q}.
Comparison of Fig.~\ref{fig:vsS-vd-7} and \ref{fig:vsS-vd-7-Q-2} shows that given the same transitional velocity $v_D$, the larger value of flux ($Q/h=10^{-2}{\rm m^2/sec}$ on Fig.~ \ref{fig:vsS-vd-7-Q-2}) results in the decreasing in the pre-Darcy effect. PI stays similar to the Darcy PI for approximately $s<0.5$. The Forchheimer effect is also seen  for approximately $s<0.65$,  causing the FDpD PI be slightly smaller than DDpD PI. Again, the particular numbers can be seen in Table~\ref{tab:PIvsS-Q}.\\
Table \ref{tab:PIvsS-vD} further explores the impact of transitional velocity $v_D$ on pre-Darcy power $s$ for flux $Q/h=10^{-4}{\rm m^2/sec}$.
  \begin{table*}
 \centering
 \begin{tabular}{c  | l l | c}
    power $s$& DDpD & FDpD & flux $Q/h$\\
  \hline\hline
  \multirow{2}{*}{$0$}
    &\multirow{2}{*}{0.1358 } &0.1358  & 1e-4\\
     &   & 0.1335 & 1e-2 \\\hline
  \multirow{2}{*}{$0.1$}
     & \multicolumn{2}{|c|}{0.0806}  & 1e-4\\
     &  0.1358  & 0.1335  & 1e-2  \\\hline
  \multirow{2}{*}{$0.3$}
    & \multicolumn{2}{|c|}{5.21e-3}& 1e-4\\
    &   0.1357  & 0.1334 & 1e-2 \\\hline
  \multirow{2}{*}{$0.5$}
    & \multicolumn{2}{|c|}{1.76e-4}& 1e-4\\
   &   0.1331  & 0.1309   & 1e-2\\\hline
  \multirow{2}{*}{$0.7$}
   & \multicolumn{2}{|c|}{5.65e-06} & 1e-4\\
    &   0.0873 & 0.0864& 1e-2 \\\hline
  \multirow{2}{*}{$1$}
   & \multicolumn{2}{|c|}{3.1e-08} & 1e-4\\
    & \multicolumn{2}{|c|}{1.66e-3}   & 1e-2\\
 \end{tabular} 
 \vspace{.1cm}
 \caption{PI for different values of power $s$ and flux $Q$ for $v_D=10^{-7}{\rm m/sec}$. }
 \label{tab:PIvsS-Q}
\end{table*}

 \begin{table*}
 \centering
 \begin{tabular}{c  | l l | c}
    power $s$& DDpD & FDpD & vel $v_D$\\
  \hline     \hline
  \multirow{1}{*}{$0$}
  & 0.1359 & 0.1358 & 1e-9, -7, -6\\
     \hline
  \multirow{3}{*}{$0.1$}
     & 0.1359 & 0.1358 & 1e-9\\
     & \multicolumn{2}{|c|}{0.0806} & 1e-7\\
     & \multicolumn{2}{|c|}{0.04941} & 1e-6  \\     \hline
  \multirow{3}{*}{$0.2$}
  & \multicolumn{2}{|c|}{0.1358} & 1e-9\\
   & \multicolumn{2}{|c|}{0.025 } & 1e-7\\
   &  \multicolumn{2}{|c|}{0.012 } & 1e-6 \\     \hline
  \multirow{3}{*}{$0.3$}
  & \multicolumn{2}{|c|}{0.1354}& 1e-9\\
    & \multicolumn{2}{|c|}{5.211e-3}& 1e-7\\
    &  \multicolumn{2}{|c|}{2.5398e-3}& 1e-6 \\     \hline
  \multirow{3}{*}{$0.5$}
  & \multicolumn{2}{|c|}{0.1128}& 1e-9\\
    & \multicolumn{2}{|c|}{1.76052e-4}& 1e-7\\
   &   \multicolumn{2}{|c|}{1.01204e-4}  & 1e-6\\     \hline
  \multirow{3}{*}{$0.7$}
  & \multicolumn{2}{|c|}{9.064e-3}& 1e-9\\
   & \multicolumn{2}{|c|}{ 5.652e-06}& 1e-7\\
    &   \multicolumn{2}{|c|}{3.774e-06}& 1e-6 \\     \hline
  \multirow{3}{*}{$1$}
  &\multicolumn{2}{|c|}{ 1.682e-05 }& 1e-9\\
   & \multicolumn{2}{|c|}{3.105e-08} & 1e-7\\
    & \multicolumn{2}{|c|}{2.446e-08}& 1e-6\\
 \end{tabular} 
 \vspace{.1cm}
 \caption{PI for different values of $s$ and $v_D$ for $Q/h=10^{-4} {\rm m^2/sec}$. }
 \label{tab:PIvsS-vD}
\end{table*}

\subsection{Dependence of PI on critical velocity $v_F$}
Here we take   $v_D=10^{-8}\ {\rm m/sec}$ to illustrate the dependence of PI on critical velocity $v_D$ in Figures~\ref{fig:vsVF-s=0.3}, for $s=0.3$, and \ref{fig:vsVF-s=0.7}, for $s=0.7$. We also take the smaller size of the reservoir $r_e=100 {\rm m}$.

\subsection{Dependence of PI on critical velocity $v_D$}
Table~\ref{tab:PIvsvD} and  Figures~\ref{fig:vsVD-s=0.05}, \ref{fig:vsVD-s=0.3} illustrate the dependence of PI on  $v_D$, the transitional velocity between Darcy and pre-Darcy flow. As usual we consider the flux $Q/h=10^{-4}{\rm m^2/sec}$ and Forchheimer transitional velocity $v_F=10^{-5}{\rm m/sec}$. To underline the impact of $v_D$ the smaller reservoir is considered with $r_e=100{\rm m}$.
The $v_D=0$ means that there is no slow flow zone and the FDpD PI is the same as FDD PI, equal 0.1976.  If  $v_D=v_F=$1e-5, then there is now medium flow (Darcy) zone, and the FDpD PI is the same as FpDpD PI: 0.0051 for $s=0.3$, Fig.~\ref{fig:vsVD-s=0.3}, and 0.1208 for $s=0.05$, Fig.~\ref{fig:vsVD-s=0.05}. The light-blue line and the last column in Table~\ref{tab:PIvsvD} represent  the case when all the reservoirs is described by the pre-Darcy equation, and the corresponding PI is the smallest.

 \begin{table*}
 \centering
 \begin{tabular}{c  | l l l l l l || c}
          $s \backslash  v_D$&\multicolumn{1}{|c}{0}  &  \multicolumn{1}{c}{5e-07}  & \multicolumn{1}{c}{1.5e-6} &  \multicolumn{1}{c}{5e-6} & \multicolumn{1}{c}{9.5e-6}  &\multicolumn{1}{c||}{1e-5}&preDarcy\\
          \hline
$0.05$ & \multirow{2}{*}{0.1976 } & 0.1754 & 0.1502& 0.1296& 0.1214&0.1208&0.1058\\
$0.3$ & & 0.0173  & 0.0084& 0.0058& 0.0051& 0.0051&0.0042\\
 \end{tabular} 
 \vspace{.1cm}
 \caption{FDpD PI dependent on $v_D$, for $Q/h=10^{-4} {\rm m^2/sec}$, $v_F10^{-5} {\rm m/sec}$, $r_e=100 {\rm m}$. }
 \label{tab:PIvsvD}
\end{table*}

\section{Experimental Results for pre-Darcy flow}\label{sec:experiment}

The aim of this section is to show the existence of deviations from linear Darcy law below a certain flow velocity ($v_D$). This pre-Darcy effect is investigated experimentally by conducting multiple experiments on real field sandstone core samples. The purpose of the experiment was also to estimate the extent of the pre-Darcy effect (the power $s$) and the range of its significance (velocity range). 

Experiments were conducted by employing the U-tube apparatus described by \cite{Fishel}, to apply low pressure gradients on the core sample. The schematic of the apparatus is shown in Fig.~\ref{fig:apparatus}. The U-tube type apparatus applies the pressure gradient on a saturated core sample through the difference of fluid levels in the columns. This difference in head is then converted to pressure gradient from the knowledge of density of the working fluid and the length of the core sample. The flow rate is calculated by noting the change in level with respect to time. This rate is then converted to superficial velocity by knowing the geometry of the setup and the core sample.

\begin{figure}
\centering
	\includegraphics[width=7cm]{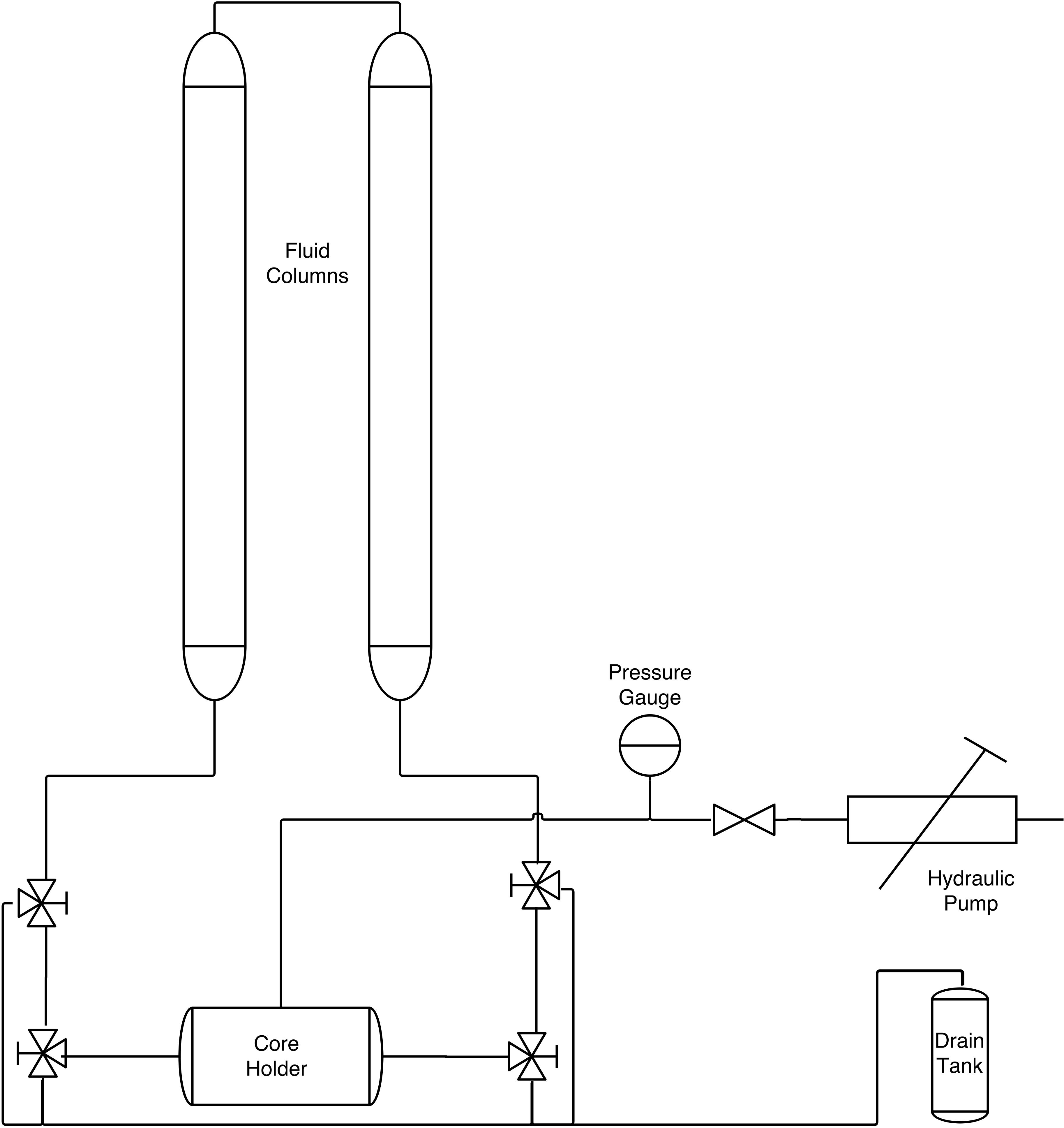}
	\vspace{0.0cm}
	\caption{Schematic of the U-tube type apparatus described by \cite{Fishel}}
	\label{fig:apparatus}
\end{figure}

Consolidated core samples from sandstone reservoirs of various permeabilities were used to study the Pre-Darcy effect in these experiments. All of the experiments point to the existence of pre-Darcy effect. The results of two such experiments are summarized in  Fig.~\ref{fig:experiment} for two different permeabilities $k=34{\rm mD}$ (triangles $\blacktriangle$) and $k=1{\rm mD}$ (squares $\square$).

For larger values of the velocities $v>v_D$, the dependence of pressure gradient $\n p$ on velocity $v$ is linear Darcy, which is as anticipated. However, for smaller values $v<v_D$ the flow exhibits the pre-Darcy effect deviating from the linear relationship. For $k=34{\rm mD}$ the value of power $s=0.6562$, and for $k=1{\rm mD}$  it is $s=0.5772$. (Note that on Fig.~\ref{fig:experiment} the pre-Darcy relation is written in the form $|\n p|=\lambda v^{-s+1}$, see \eqref{eq:g-v-np} and \eqref{eq:prd}.) The transitional velocity $v_D$ appears to be in the range $[$1e-8, 1e-7$]$m/s. 

The results confirm the presence of pre-Darcy effect at low velocities. These experiments also provide the range of velocity values and the parameter $s$ used in this study for calculating the productivity index. 

\begin{figure}
\centering
	\includegraphics[width=11cm]{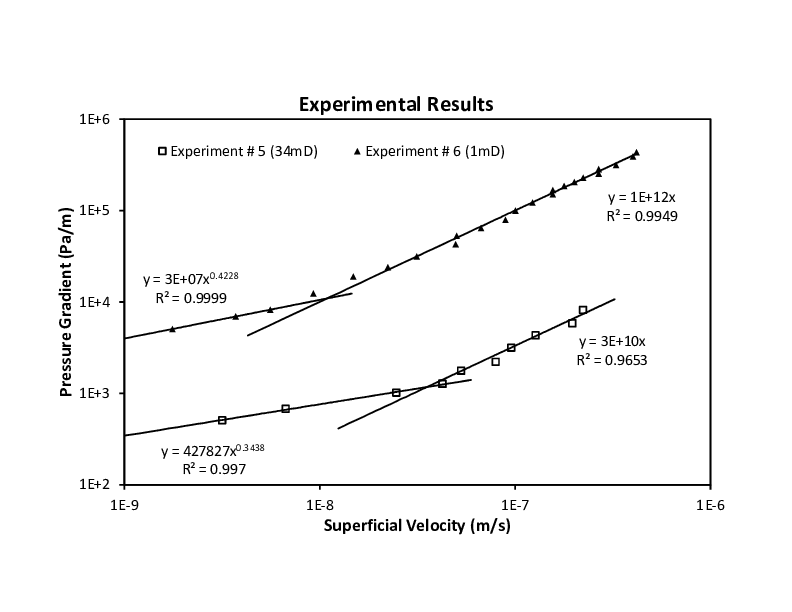}
	\vspace{-0.6cm}
	\caption{Experimental results plotted on a log-log plot of pressure gradient vs velocity. $k=34{\rm mD}$ (triangles $\blacktriangle$) and $k=1{\rm mD}$ (squares $\square$)}
	\label{fig:experiment}
\end{figure}

\section{Conclusions}
In this paper we explore the impact of the nonlinearity of the flow of slightly compressible fluid on the value of PI of the well. The single mathematical formulation is used to combine three flow regimes, pre-Darcy, Darcy and post-Darcy dependent on the critical transitional velocities.
The existence of pre-Darcy effect for small flow velocities is confirmed by our experimental results obtained for several permeability distributions. This agrees with the results obtained by other researchers. Our  analytical formulas for  PI for radial symmetric flows and consequent computations  show the significant impact of pre-Darcy effect on the value of PI and thus on management of the reservoir-well system and hydrocarbon extraction. 
\begin{itemize}
 \item Dependence on flux $Q$: It is a well-known fact that the Darcy PI does not depend on the production rate $Q$ and Forchheimer PI decreases to zero with increasing $Q$. Pre-Darcy PI is essentially zero for small values of $Q$ even in PSS case. FDpD PI then behaves as pre-Darcy PI for the small values of $Q$, and as Forchheimer PI for the large. 
 \item Dependence on pre-Darcy power $s$: The  pre-Darcy PI is the same as Darcy PI for $s=0$. With $s$ increasing to 1, the PI with pre-Darcy component (DDpD, FDpD, pre-Darcy) decreases. For smaller $Q$ it decreases faster. The difference between DDpD and FDpD PI disappears for larger values of $s$.
 \item Dependence on the transitional velocity from Darcy to pre-Darcy regime: FDpD PI decreases with increasing critical velocity (dominating slow pre-Darcy zone). The larger is the power $s$ the more dramatic is the drop of the PI. 
\end{itemize}

\section*{Acknowledgments}
The research of the second author was partially supported by 
the NSF grant DMS-1412796.

\bibliographystyle{elsarticle-num}
\bibliography{reference}



\begin{figure}
  \centering
  \includegraphics[width=0.7\textwidth]{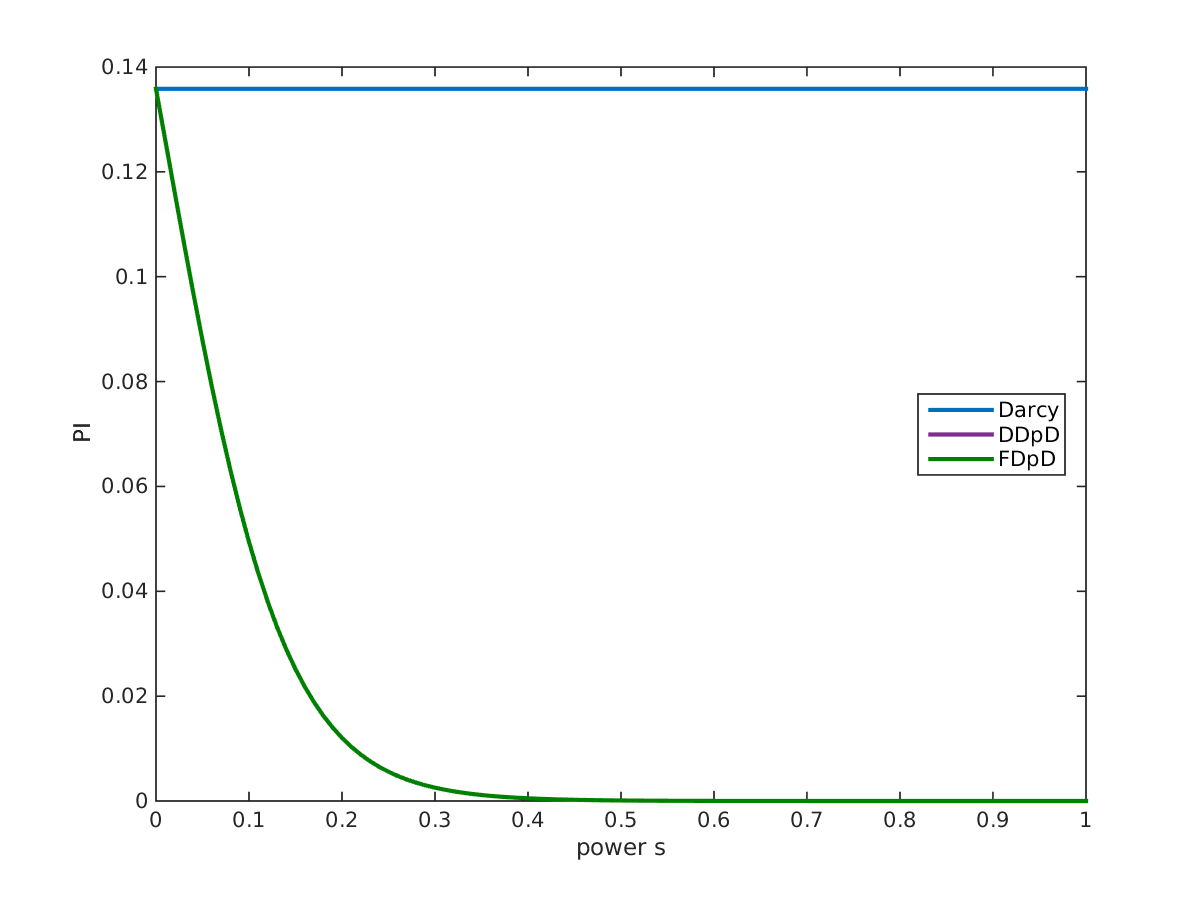}
  \vspace{-0.6cm}
  \caption{Dependence   on power $s$ for $Q/h=10^{-4}{\rm m^2/sec}$, $v_D=10^{-6}{\rm m/sec}$}
  \label{fig:vsS-vd-6}
 \end{figure}
 \begin{figure}
   \centering
  \includegraphics[width=0.7\textwidth]{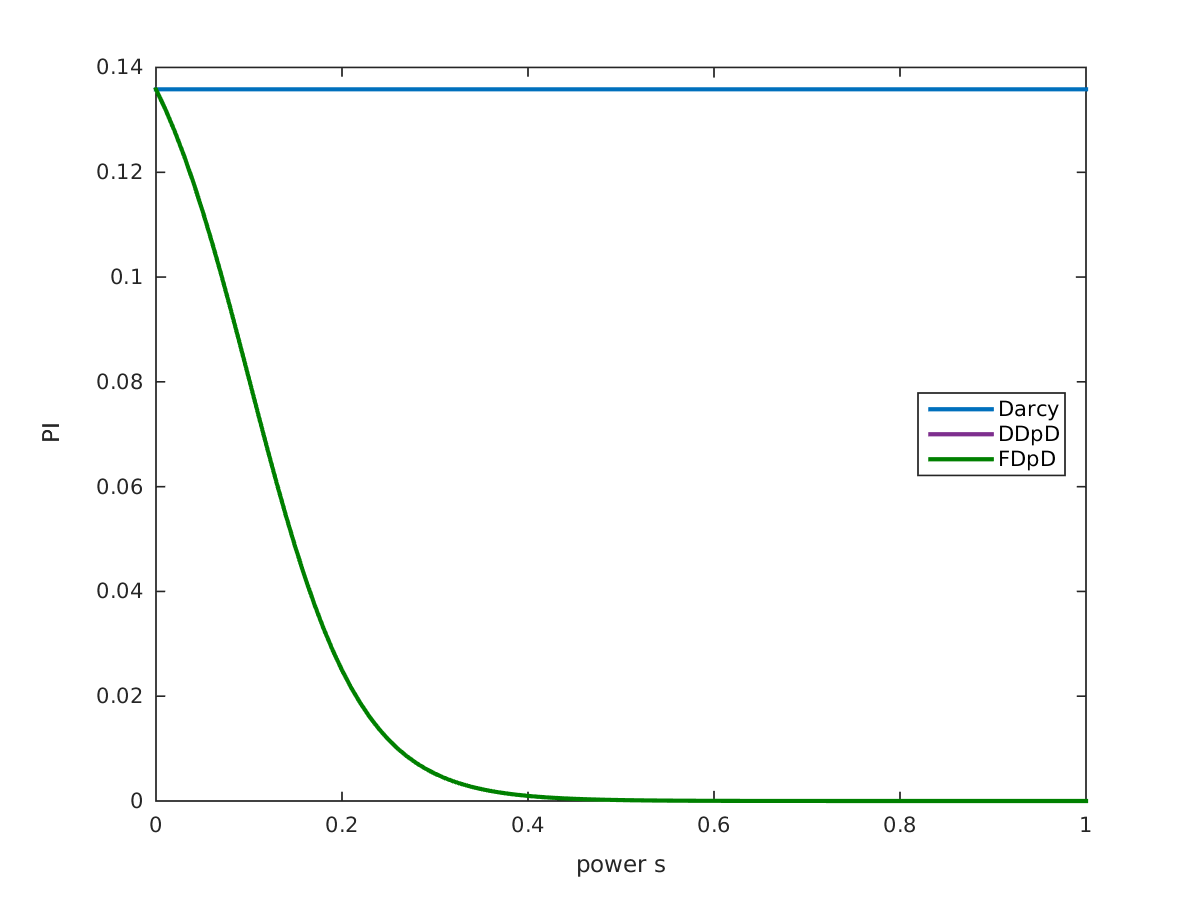}
    \vspace{-0.6cm}
  \caption{Dependence  on power $s$ for $Q/h=10^{-4}{\rm m^2/sec}$, $v_D=10^{-7}{\rm m/sec}$}
    \label{fig:vsS-vd-7}
  \end{figure}
 \begin{figure}
   \centering
    \includegraphics[width=0.7\textwidth]{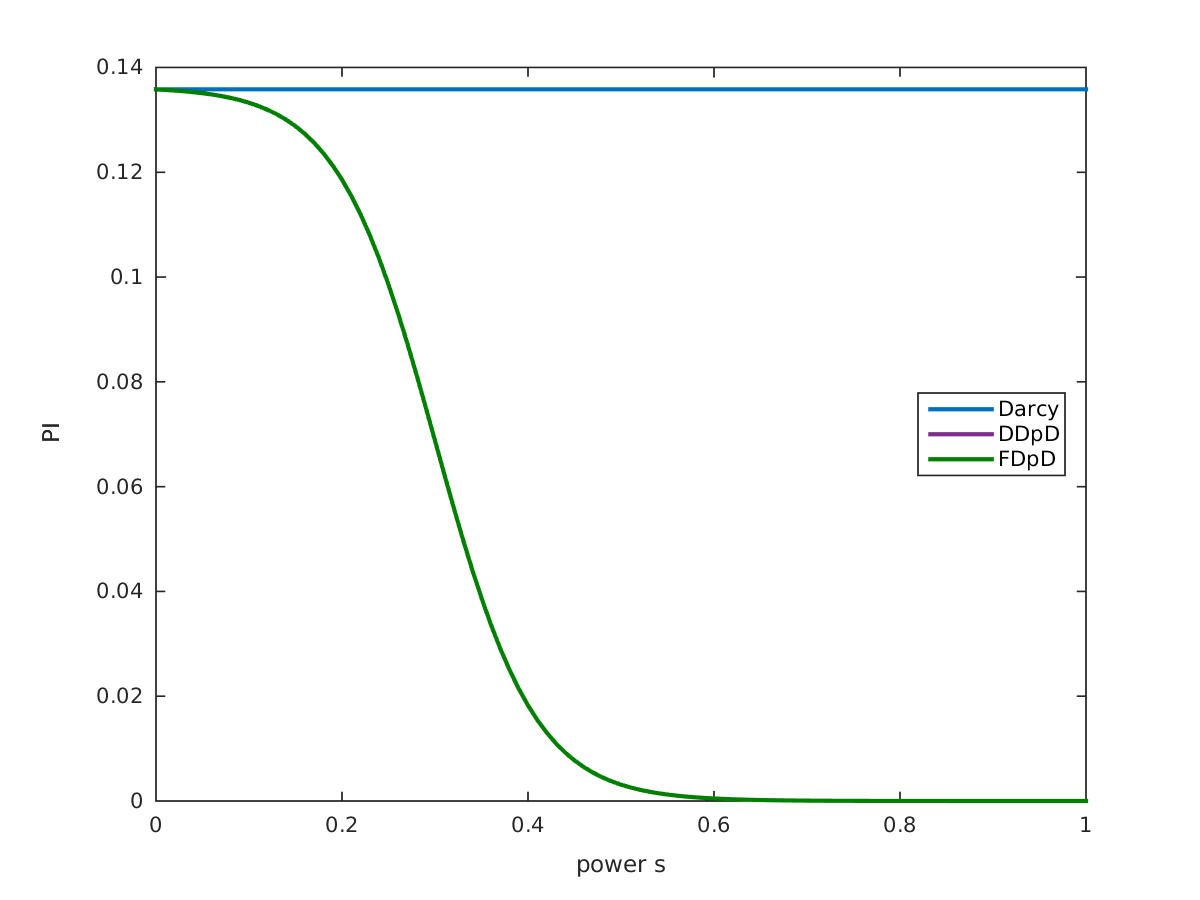}
    \vspace{-0.6cm}
  \caption{Dependence   on power $s$ for $Q/h=10^{-4}{\rm m^2/sec}$, $v_D=10^{-8}{\rm m/sec}$}
    \label{fig:vsS-vd-8}
  \end{figure}
 \begin{figure}
   \centering
  \includegraphics[width=0.7\textwidth]{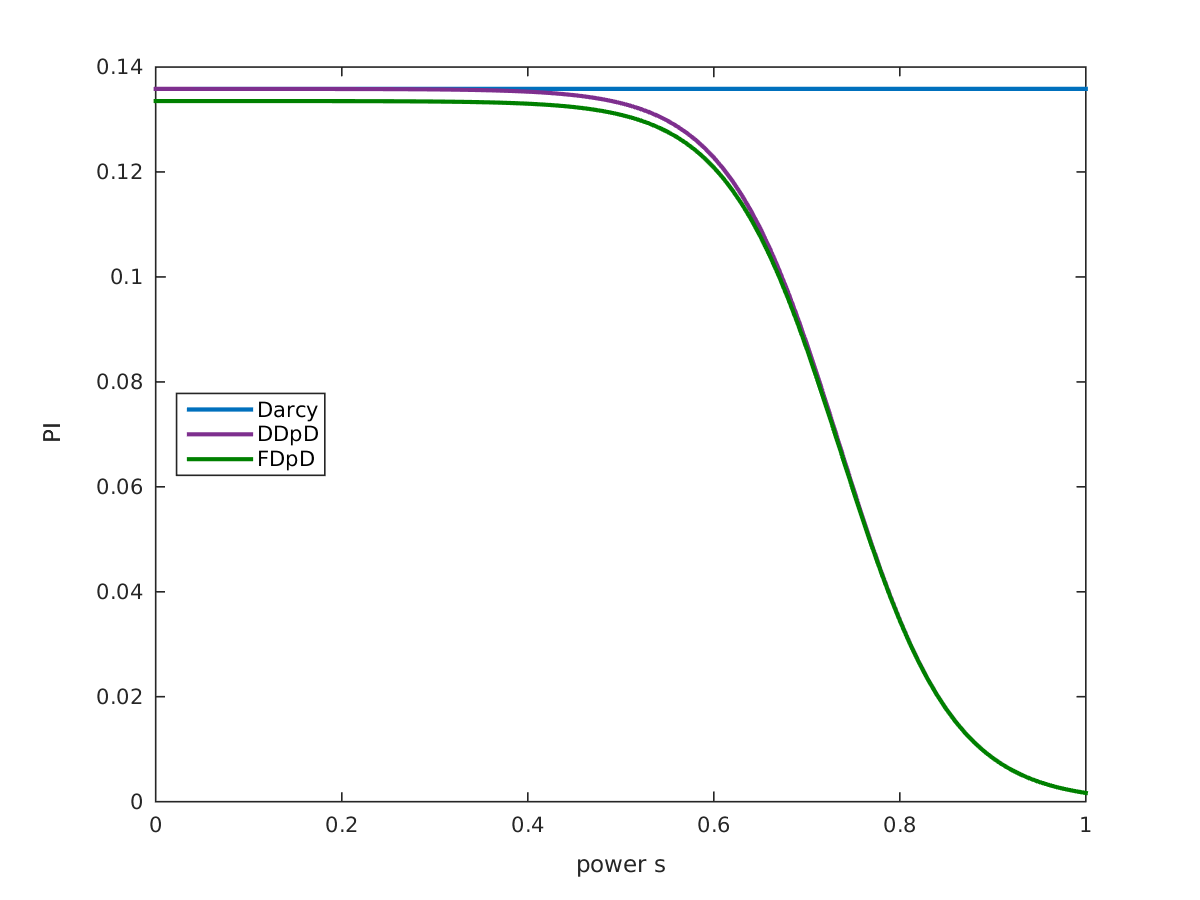}
  \vspace{-0.6cm}
  \caption{Dependence   on power $s$ for $Q/h=10^{-2}{\rm m^2sec}$, $v_D=10^{-7}{\rm m/sec}$}  
    \label{fig:vsS-vd-7-Q-2}
\end{figure}


\begin{figure}
  \centering
  \includegraphics[width=0.7\textwidth]{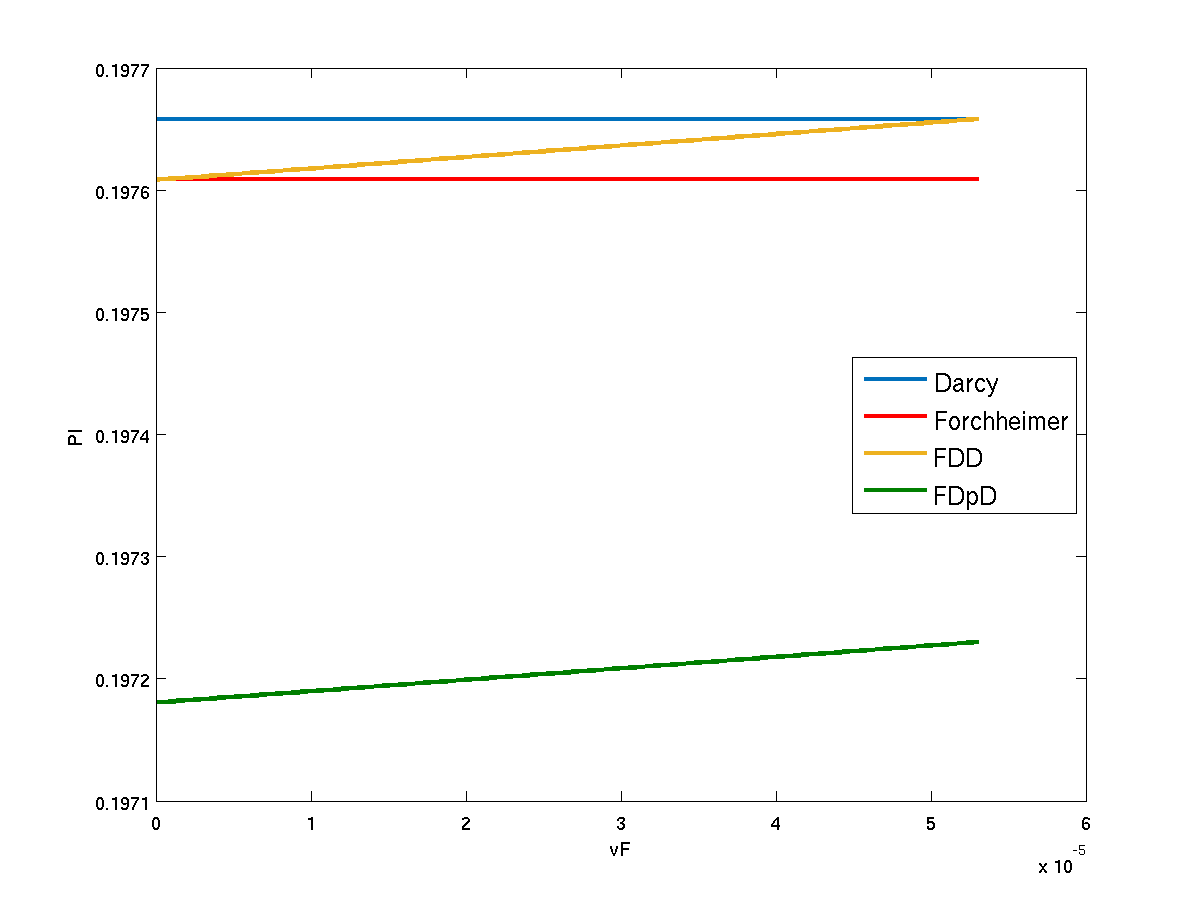}
  \vspace{-0.6cm}
  \caption{Dependence  on $v_F$ for $s=0.3$, $r_e=100{\rm m}$}
  \label{fig:vsVF-s=0.3}
\end{figure}  
\begin{figure}
  \centering
  \includegraphics[width=0.7\textwidth]{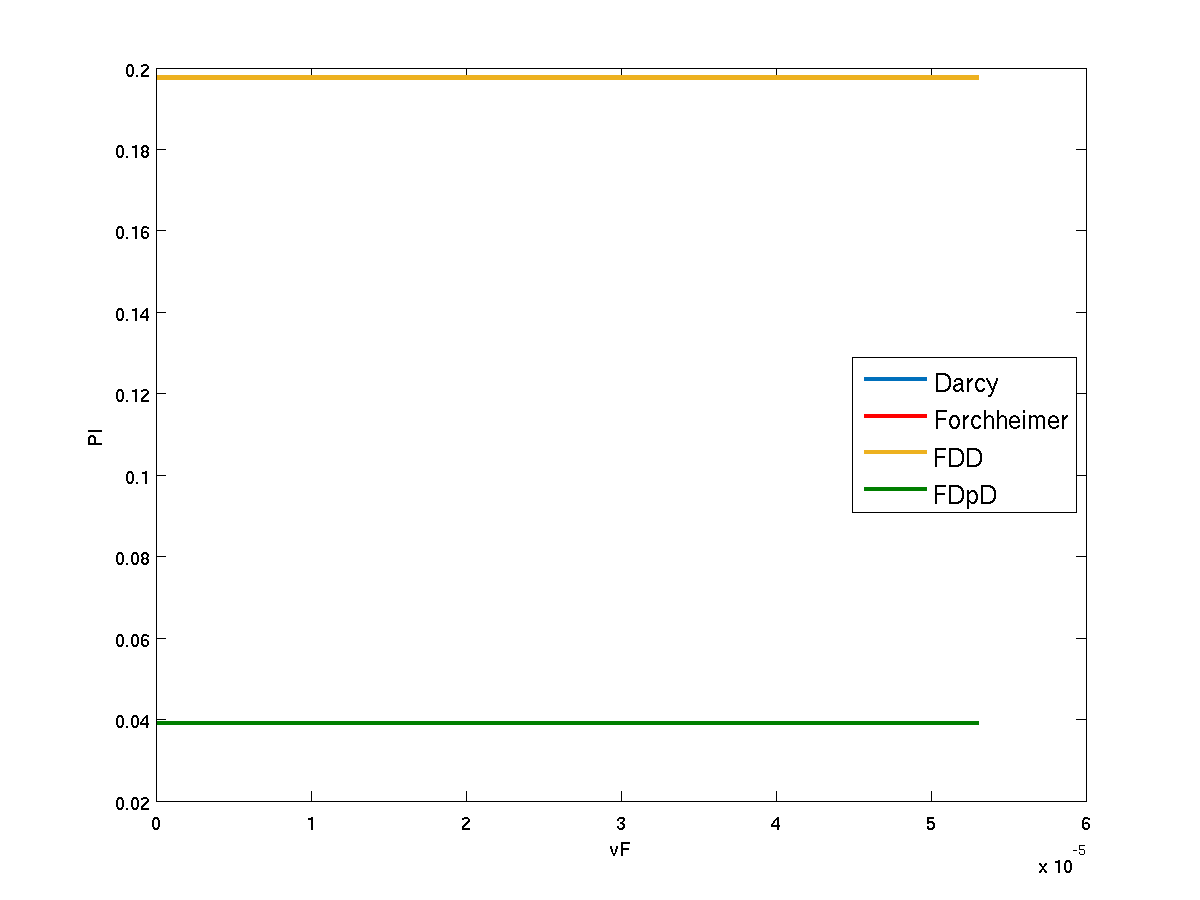}
    \vspace{-0.6cm}
  \caption{Dependence  on $v_F$ for  $s=0.7$, $r_e=100{\rm m}$}
  \label{fig:vsVF-s=0.7}
\end{figure}


\begin{figure}
  \centering
  \includegraphics[width=0.7\textwidth]{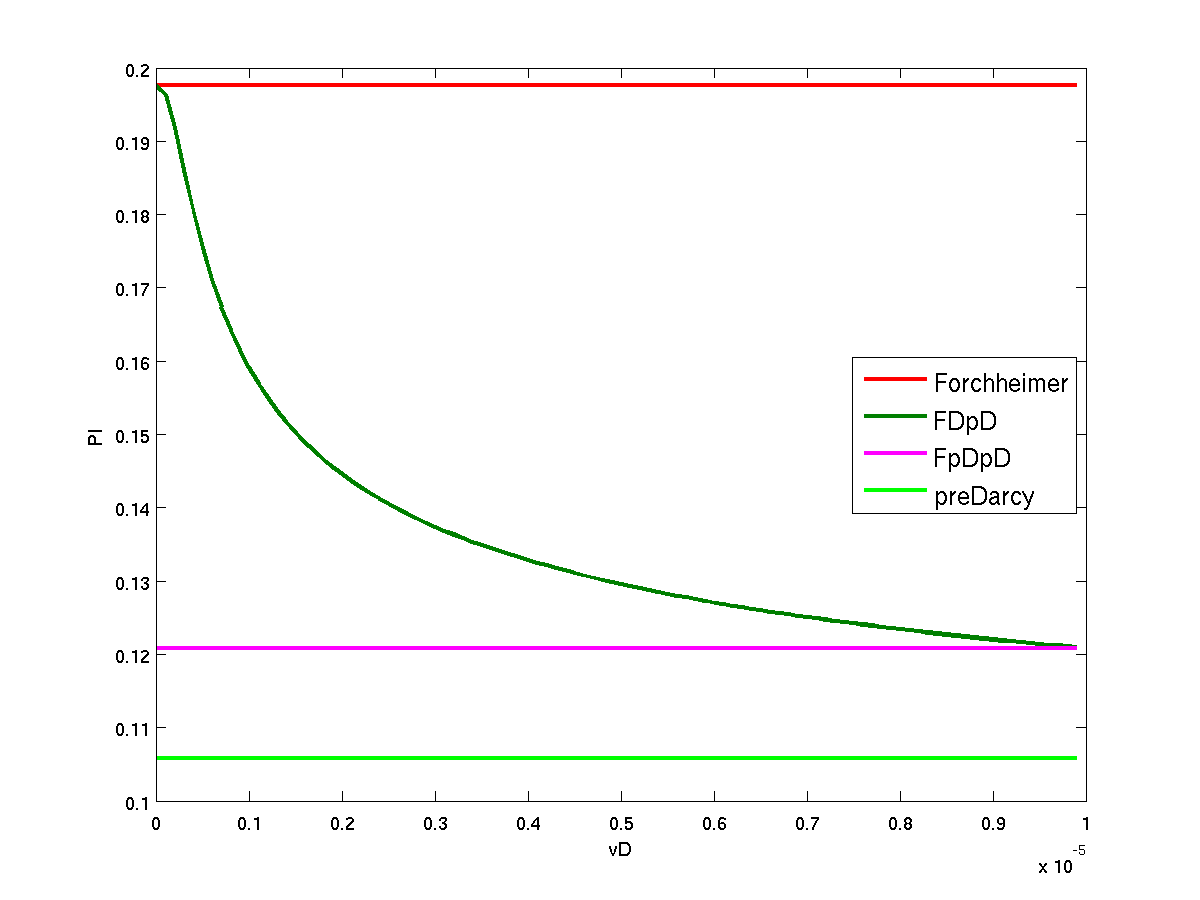}
  \vspace{-0.6cm}
  \caption{Dependence  on $v_D$ for  $s=0.05$, $r_e=100{\rm m}$}
    \label{fig:vsVD-s=0.05}
\end{figure}  

\begin{figure}
  \centering
  \includegraphics[width=0.7\textwidth]{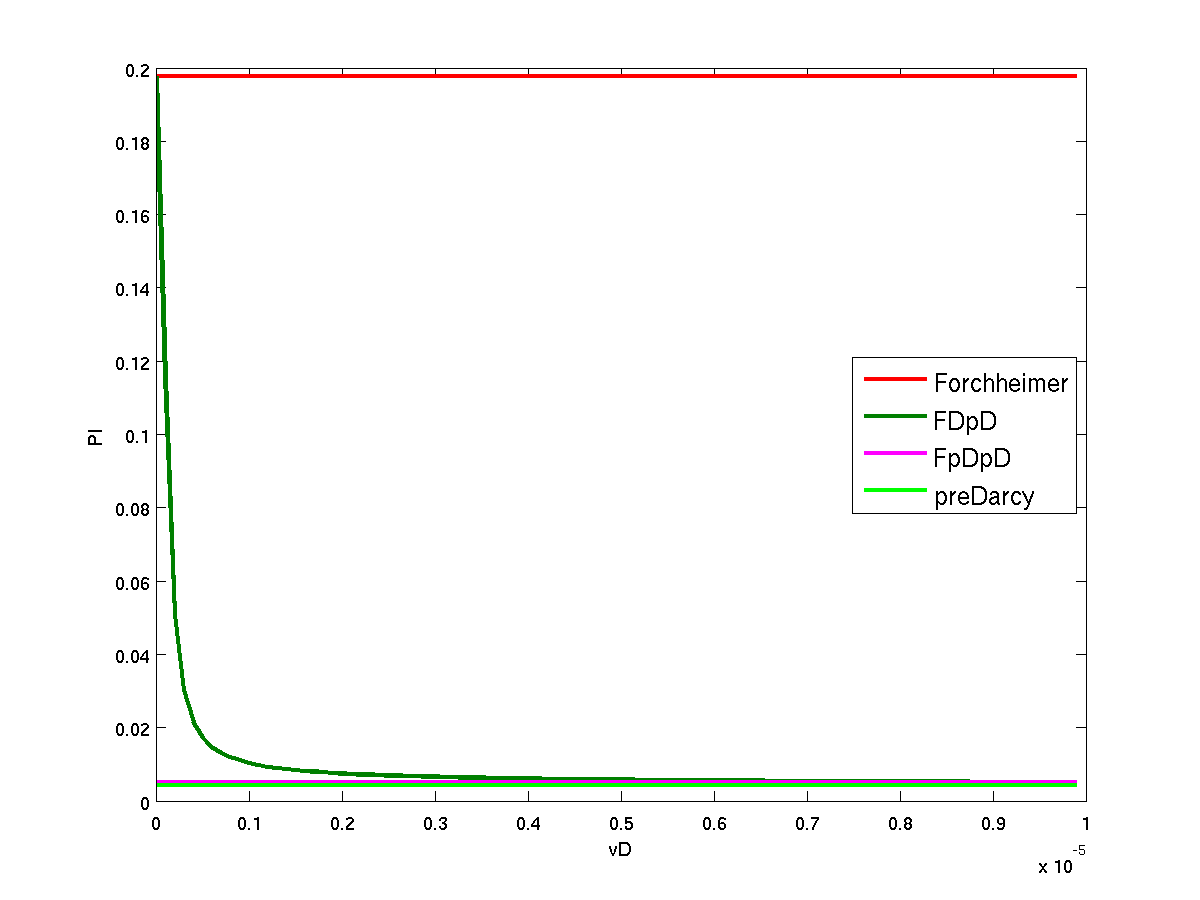}
    \vspace{-0.6cm}
  \caption{Dependence on $v_D$ for $s=0.3$, $r_e=100{\rm m}$}
  \label{fig:vsVD-s=0.3}
\end{figure}

\end{document}